\documentclass[letterpaper, 10 pt, conference]{ieeeconf}  

\IEEEoverridecommandlockouts                              
\usepackage{graphics} 
\usepackage{amsmath} 
\usepackage{amssymb}  
\usepackage{amsthm}

\usepackage[colorlinks=true]{hyperref}

\usepackage{cite}
\usepackage{graphicx}
\usepackage{todonotes}

\newtheorem{thm}{Theorem}[section]
\newtheorem{assum}[thm]{Assumption}
\newtheorem{lem}[thm]{Lemma}
\newtheorem{cor}[thm]{Corollary}
\numberwithin{equation}{section}

\title{\LARGE \bf Nested Distributed Gradient Methods with Stochastic Computation Errors
}

\author{Charikleia Iakovidou$^{1}$ and Ermin Wei$^{2}$
\thanks{*This work was supported by the DARPA award HR-001117S0039.}
\thanks{{C. Iakovidou  and E. Wei are with the Department of Electrical and Computer Engineering, Northwestern University,
        Evanston, IL USA,  
        {\tt\small chariako@u.northwestern.edu, ermin.wei@northwestern.edu}}}
}

\begin{document}

\maketitle
\thispagestyle{empty}
\pagestyle{empty}

\begin{abstract}

In this work, we consider the problem of a network of agents collectively minimizing a sum of convex functions. The agents in our setting can only access their local objective functions and exchange information with their immediate neighbors. Motivated by applications where computation is imperfect, including, but not limited to, empirical risk minimization (ERM) and online learning, we assume that only noisy estimates of the local gradients are available. To tackle this problem, we adapt a class of Nested Distributed Gradient methods (NEAR-DGD) to the stochastic gradient setting. These methods have minimal storage requirements, are communication aware and perform well in settings where gradient computation is costly, while communication is relatively inexpensive. We investigate the convergence properties of our method under standard assumptions for stochastic gradients, i.e. unbiasedness and bounded variance. Our analysis indicates that our method converges to a neighborhood of the optimal solution with a linear rate for local strongly convex functions and appropriate constant steplengths. We also show that distributed optimization with stochastic gradients achieves a noise reduction effect similar to mini-batching, which scales favorably with network size. Finally, we present numerical results to demonstrate the effectiveness of our method.

\end{abstract}

\section{INTRODUCTION}

Consider the setting where a group of agents (nodes) in a connected network coordinate to minimize a global objective
\begin{equation}
    \label{eq:prob_cen}
    \min_{x \in \mathbb{R}^p} f(x) = \sum_{i=1}^n f_i(x),
\end{equation}
where $f:\mathbb{R}^p \rightarrow \mathbb{R}$ is the global objective function, \mbox{$f_i:\mathbb{R}^p \rightarrow \mathbb{R}$} the local objective function accessed by agent $i \in \{1,...,n\}$ and $x \in \mathbb{R}^p$ the global decision variable. This formulation frequently arises in applications such as sensor networks \cite{sensor1,sensor2}, multi-vehicle systems \cite{vehicles} and smart grids \cite{smart_grids}, to name a few.

Every agent~$i$ needs to maintain a local estimate $x_i \in \mathbb{R}^p$ of the global variable $x$. On account of this, problem~\ref{eq:prob_cen} is commonly reformulated as \cite{bertsekas1989parallel}
\begin{equation}
    \label{eq:prob_distr_1}
    \begin{split}
        \min_{x_i \in \mathbb{R}^p} &\sum_{i=1}^n f_i(x_i)\text{, s.t. } x_i = x_j \text{ if $j \in \mathcal{N}_i$},
    \end{split}
\end{equation}
where $\mathcal{N}_i$ is the neighborhood of node $i$. 

Solving problem~\ref{eq:prob_distr_1} is a topic well-studied in literature. Common approaches include first order methods \cite{NedicSubgradientConsensus,nedic_constrained_2010,MouraFastGrad,near_dgd_2017,extra}, primal-dual algorithms \cite{zhu_distr_conv_2012,KIA2015254}, gradient tracking methods \cite{di2016next,qu2017harnessing,diging,push_pull}, the alternating direction method of multipliers (ADMM) \cite{ADMMBoyd} and Newton methods \cite{mokhtari_newton,samira_newton}. 

In recent years, the proliferation of available data has attracted significant interest in distributed systems for machine learning \cite{tsianos_ml_2012,predd_ml_2009,Bekkerman:2011:SUM:2124408,Boyd:2011:DOS:2185815.2185816,duchi_dual_2012,shen_towards_2018}. Moreover, the increasing cost of gradient computation in large-scale systems has motivated the cheaper alternative of stochastic gradient descent (SGD) \cite{sgd_bottou,Bach:2011:NAS:2986459.2986510}. There is an extensive body of work on the combination of distributed systems and stochastic gradients for general network topologies \cite{sundharram_distributed_2010,bianchi_convergence_2011,srivastava_distributed_2011,duchi_dual_2012,nedic_stochastic_2014,towfic_adaptive_2014,shamir_distributed_2014,mokhtari_dsa:_2015,vanli_stochastic_2014,sirb_decentralized_2016,chatzipanagiotis_distributed_2016,lan_communication-efficient_2017,hong_stochastic_2017,morral_success_2017,pu_flocking-based_2017,bijral_data-dependent_2016,lian_can_nodate,lian_asynchronous_2017,nokleby_distributed_2017,pu_distributed_2018-1,jakovetic_convergence_2018,tang_d$^2$:_2018,shen_towards_2018,sahu_communication-efficient_2018,olshevsky_robust_2018,yuan_variance-reduced_2019,xin_distributed_2019,olshevsky_non-asymptotic_2019}. A number of these works have demonstrated that distributed stochastic algorithms are comparable or even outperform their centralized counterparts in some cases  \cite{shamir_distributed_2014,bijral_data-dependent_2016,morral_success_2017,pu_flocking-based_2017,lian_can_nodate,pu_distributed_2018-1,olshevsky_robust_2018,olshevsky_non-asymptotic_2019}. An interesting attribute of some distributed stochastic methods is that they achieve a variance reduction effect similar to mini-batching \cite{pu_flocking-based_2017,pu_distributed_2018-1,olshevsky_robust_2018,olshevsky_non-asymptotic_2019,lian_can_nodate,shen_towards_2018}. However, only a small subset of methods can achieve linear convergence \cite{mokhtari_dsa:_2015,pu_distributed_2018-1,shen_towards_2018,yuan_variance-reduced_2019,olshevsky_robust_2018,xin_distributed_2019}. In addition, none of these methods is adaptable to varying application conditions and some of them have excessive memory requirements \cite{mokhtari_dsa:_2015,shen_towards_2018} or rely on data reshuffling schemes \cite{yuan_variance-reduced_2019}.

A class of Nested Distributed Gradient methods (NEAR-DGD) was first introduced in \cite{near_dgd_2017}. The power of \mbox{NEAR-DGD} lies in its versatile framework, which alternates between gradient steps and an adjustable number of nested consensus steps. Some variants of NEAR-DGD provably achieve exact linear convergence. Furthermore, NEAR-DGD does not rely on previous gradients or more than one previous iterates and has therefore minimal storage requirements. In this work, we develop a class of \mbox{NEAR-DGD} based methods that utilize stochastic gradient approximations. NEAR-DGD is particularly suited for settings where the cost of computation is high, and the use of stochastic gradients could alleviate computational load even further. We add to the existing methods that achieve linear convergence rates and the methods that produce a variance reduction effect proportional to network size.

The rest of this paper is organized as follows. We list the assumptions for our analysis at the end of the current section. In Section~\ref{sec:algo}, we briefly summarize the original and the stochastic NEAR-DGD methods. We derive the convergence properties of the stochastic NEAR-DGD method in Section~\ref{seq:anal}. Finally, in Section~\ref{sec:res}, we present our numerical results and in Section~\ref{sec:final} we conclude this work.

\subsection{Notation}

For the rest of this paper, a local variable at agent $i$ and at iteration count $k$ will be denoted with $v_{i,k} \in \mathbb{R}^p$. The concatenation of all local variables $\mathbf{v}_k=[v_{1,k},...,v_{n,k}]$ will be denoted with lowercase boldface letters, while uppercase boldface letters are reserved for matrices. We will refer to the identity matrix of dimension $p$ as $I_p$ and the vector of ones of dimension $n$ as $1_n$. The notation \mbox{$\mathcal{M}_n:=\frac{1}{n}(1_n1_n^T \otimes I_p)$} denotes the operator that returns the average vector $\bar{v}_k$ across $n$ agents, i.e. $\bar{v}_k = \mathcal{M}_n \mathbf{v}_k = \frac{1}{n}\sum_{i=1}^n v_{i,k}$.

\subsection{Assumptions}
We make the following standard assumption for the local objective functions $f_i$ throughout our analysis.
\begin{assum} \label{assum_fi_mu_L} \textbf{(Local strong convexity and Lipschitz gradients)}
Each local objective function $f_i$ is $\mu_i$-strongly convex and $L_i$-smooth.
\end{assum}
We also assume that agents are unable to compute the true local gradients $\nabla f_i$, but have access to the stochastic gradient estimates $g_i$ that satisfy the following condition.
\begin{assum}\label{assm:err_stoch}\textbf{(Stochastic gradients)}
	Let $g_i(x_{i,k},\xi_{i,k}) \in \mathbb{R}^p$ be the stochastic gradient computed at agent $i$ at iteration $k$ and $\xi_{i,k}$ a random vector. Then for all \mbox{$i=1,...,n$} and $k=0,1,2,...$, $g_i$ is an unbiased estimator of the true local gradient $\nabla f_i(x_{i,k})$ and its variance is bounded, namely
	\begin{gather*}
	    \mathbb{E}[g_i(x_{i,k},\xi_{i,k})| x_{i,k}] = \nabla f_i(x_{i,k}),\\ \mathbb{E}[\|g_i(x_{i,k},\xi_{i,k})-\nabla f_i(x_{i,k})\|^2 | x_{i,k}] \leq \sigma^2,
	\end{gather*}
for some $\sigma>0$.
\end{assum}
Common examples of gradient estimators that satisfy Assumption~\ref{assm:err_stoch} include stochastic and mini-batch gradients.
\section{Algorithm}
\label{sec:algo}

Before introducing the stochastic NEAR-DGD method, we will first reformulate problem~\ref{eq:prob_distr_1} as
\begin{equation}
    \label{eq:prob_distr_2}
    \begin{split}
        \min_{x_i \in \mathbb{R}^{p}} &\sum_{i=1}^n f_i(x_i)\text{, s.t. } (\mathbf{W} \otimes I_p) \mathbf{x} = \mathbf{x},
    \end{split}
\end{equation}
where $\mathbf{W} \in \mathbb{R}^{n \times n}$ is a matrix with the following properties.
\begin{assum}\label{assm:w} \textbf{(Consensus matrix)}
The matrix $\mathbf{W}$ is symmetric, doubly stochastic and its diagonal elements are strictly positive, i.e. $w_{ii} > 0$. Its off-diagonal elements $w_{ij}$ are strictly positive if and only if agents $i$ and $j$ are connected and zero otherwise.
\end{assum}
We will refer to $\mathbf{W}$ as the consensus matrix. Note that $(\mathbf{W} \otimes I_p) \mathbf{x} = \mathbf{x}$ if and only if $x_i = x_j$ for all connected node pairs $i$, $j$.

\subsection{The NEAR-DGD method}

We will now briefly summarize the NEAR-DGD method first published in \cite{near_dgd_2017}. Each iteration of \mbox{NEAR-DGD} is composed of a number of successive consensus steps, during which agents communicate with their neighbors, followed by a gradient step executed locally. Starting from initial point $\mathbf{y}_0=[y_{1,0},...,y_{n,0}] \in \mathbb{R}^{np}$, the system-wide iterates $\mathbf{x}_k$ and $\mathbf{y}_k$ of NEAR-DGD can be written as
\begin{gather}
        \mathbf{x}_k =\mathbf{Z}^{t(k)}\mathbf{y}_k \label{eq:near_dgd_it_n_1_orig},\\
        \mathbf{y}_{k+1} = \mathbf{x}_k - \alpha \nabla \mathbf{f}(\mathbf{x}_k),
        \label{eq:near_dgd_it_n_2_orig}
\end{gather}
where \mbox{$\mathbf{Z}=(\mathbf{W} \otimes I_p) \in \mathbb{R}^{np \times np}$}, $t(k)$ is the number of communication rounds performed at iteration $k$ and $\nabla \mathbf{f}(\mathbf{x}_k)=[\nabla f_1(x_{1,k}),...,\nabla f_n(x_{n,k})]$ is the concatenation of the local gradients at iteration $k$.

\subsection{The stochastic NEAR-DGD method}

For most of this work, we will focus on the variant of NEAR-DGD where the number of consensus steps per iteration is constant, i.e. $t(k)=t$, for some $t>0$. This is also known as NEAR-DGD$^t$ \cite{near_dgd_2017}.

The system-wide iterates $\mathbf{x}_k$ and $\mathbf{y}_k$ of the stochastic NEAR-DGD$^t$ method can be written as
\begin{gather}
        \mathbf{x}_k = \mathbf{Z}^t\mathbf{y}_k \label{eq:near_dgd_it_n_1},\\
        \mathbf{y}_{k+1} =\mathbf{x}_k - \alpha  \mathbf{g}(\mathbf{x}_k,\pmb{\xi}_k).
        \label{eq:near_dgd_it_n_2}
\end{gather}
where $\mathbf{g}(\mathbf{x}_k,\pmb{\xi}_k) \in \mathbb{R}^{np}$ is the concatenation of the local stochastic gradients $g_i(x_{i,k},\xi_{i,k})$ which satisfy Assumption~\ref{assm:err_stoch} for all $i=1,...,n$.

\section{Convergence Analysis}
\label{seq:anal}
\subsection{Preliminaries}
In this subsection, we introduce a number of helpful lemmas necessary for our main analysis. 
\begin{lem}
\label{lem:global_mu_L}
\textbf{(Average strong convexity and Lipschitz gradients)} Let $\bar{f}(x)=\frac{1}{n}\sum_{i=1}^n f_i(x)$ be the average of the local objective functions $f_i$. It follows then that $\bar{f}$ is \mbox{$\mu_{\bar{f}}$-strongly} convex and $L_{\bar{f}}$-smooth, where $\mu_{\bar{f}}=\frac{1}{n}\sum_{i=1}^n \mu_i$ and $L_{\bar{f}}=\frac{1}{n}\sum_{i=1}^n L_i$.
\end{lem}
Notice that Lemma~\ref{lem:global_mu_L} is a direct consequence of Assumption~\ref{assum_fi_mu_L}.
The following lemma is adapted from  \cite[Theorem 2.1.11, Chapter 2]{nesterov_introductory_1998}.
\begin{lem} 		
\label{lem:conv_smooth}
	\textbf{(Strongly convex functions with Lipschitz gradients)} Let $f:\mathbb{R}^p \rightarrow \mathbb{R}$ be a $\mu$-strongly convex and $L$-smooth function. Then for all $x \in \mathbb{R}^p$, we have
	\begin{equation*}
    \langle x-x^\star,\nabla f(x) \rangle \geq \frac{\mu L}{\mu+L}\|x-x^\star\|^2 + \frac{1}{\mu+L}\|\nabla f(x)\|^2,
\end{equation*}
where $x^\star=\arg \min_x f(x)$. 
\end{lem}
In the following lemma, we show that centralized SGD converges to a neighborhood of the solution under the conditions described in Assumption~\ref{assm:err_stoch} and for appropriate steplength choices. This result is necessary for proving the convergence properties of the stochastic NEAR-DGD$^t$ method in the next subsection.
\begin{lem} 		
\label{lem:grad_err_cen}
	\textbf{(Stochastic gradient descent)} 
	Consider the problem 
	\begin{equation}
	\label{eq:prob_centr}
	    \min_{x\in\mathbb{R}^p} f(x),
	\end{equation}
	where $f$ is a $\mu$-strongly convex and $L$-smooth function and let $\{x_k\}$ be the sequence generated by the stochastic gradient method with constant steplength $\alpha$
\begin{equation*}
    x_{k+1}=x_k - \alpha g(x_k,\xi_k)
\end{equation*}
where $\xi_k$ is a random vector and $g$ is an unbiased estimator of the true gradient $\nabla f$ with bounded variance, i.e. $\mathbb{E}[g(x_{k},\xi_{k})| x_{k}] = \nabla f(x_{k})$ and $\mathbb{E}[\|g(x_{k},\xi_{k})-\nabla f(x_{k})\|^2 | x_{k}] \leq \sigma^2$, for some $\sigma>0$. Also let the steplength $\alpha$ satisfy
\begin{equation*}
    \alpha \leq \frac{2}{\mu + L},
\end{equation*}
and let $\mathcal{F}_k$ be a $\sigma$-algebra containing all the information generated by $\{x_k\}$ up to and including iteration $k$. Then for all $k=0,1,2,...$, we have 
\begin{equation}
\label{eq:grad_err_sto}
    \mathbb{E}[\|x_{k+1} - x^\star\|^2 | \mathcal{F}_k] \leq (1-2\alpha\gamma)\|x_k - x^\star\|^2 +  \alpha^2 \sigma^2,
\end{equation}
where $\gamma=\frac{\mu L}{\mu + L}$.

Thus, the stochastic gradient method converges in expectation to a $\mathcal{O}\bigg( \frac{\alpha \sigma^2}{\gamma}\bigg)$ neighborhood of the solution of problem~(\ref{eq:prob_centr}) with a linear rate.
\end{lem}
\begin{proof} Consider,
\begin{equation}
\label{eq:pf_stochg}
    \begin{split}
        \|x_{k+1}-x^\star\|^2 &= \|x_k - \alpha g(x_k,\xi_k)- x^\star\|^2 \\ 
        & = \|x_k - x^\star\|^2 +\alpha^2 \|g(x_k,\xi_k)\|^2\\ 
        &-2\alpha \langle x_k - x^\star ,g(x_k,\xi_k) \rangle.
    \end{split}
\end{equation}
For the magnitude of the stochastic gradient $g$, we have
\begin{equation*}
    \begin{split}
        \|g(x_k,\xi_k)\|^2 &= \|g(x_k,\xi_k) - \nabla f (x_k) + \nabla f (x_k) \|^2 \\
        &= \|g(x_k,\xi_k) - \nabla f(x_k)\|^2 + \|\nabla f (x_k)\|^2 \\&- 2\langle \nabla f(x_k), \nabla f(x_k) - g(x_k,\xi_k) \rangle.
    \end{split}
\end{equation*}
We furthermore have
\begin{equation*}
    \begin{split}
        \langle x_k - x^\star ,g(x_k,\xi_k) \rangle &=  \langle x_k - x^\star ,\nabla f(x_k) \rangle \\
        &+\langle x_k - x^\star , g(x_k,\xi_k) - \nabla f(x_k) \rangle.
    \end{split}
\end{equation*}
Combining all of the above and taking the conditional expectation on both sides of (\ref{eq:pf_stochg}) with respect to $\mathcal{F}_k$, yields
\begin{equation*}
    \begin{split}
        \mathbb{E}[\|x_{k+1}-x^\star\|^2 |\mathcal{F}_k]&\leq \|x_k - x^\star\|^2 + \alpha^2 \sigma^2 +\alpha^2 \|\nabla f (x_k)\|^2 \\&-2\alpha \langle x_k - x^\star ,\nabla f(x_k) \rangle.
    \end{split}
\end{equation*}
Finally, by {applying} Lemma~\ref{lem:conv_smooth} {to the last term,} we obtain
\begin{equation*}
    \begin{split}
        \mathbb{E}[\|&x_{k+1}-x^\star\|^2|\mathcal{F}_k] \leq \|x_k - x^\star\|^2 + \alpha^2\sigma^2 \\&+ \alpha^2\|\nabla f(x_k)\|^2 
         -2\alpha \gamma \| x_k - x^\star\|^2 - \frac{2\alpha}{\mu+L}\|\nabla f(x_k)\|^2.
    \end{split}
\end{equation*}
By grouping the terms together and noticing that the coefficient of $\|\nabla f (x_k)\|^2$ is non-positive due to the definition of $\alpha$, we obtain (\ref{eq:grad_err_sto}).
Applying (\ref{eq:grad_err_sto}) recursively and taking the total expectation, yields
\begin{equation*}
    \mathbb{E}[\|x_{k} - x^\star\|^2 ] \leq (1-2\alpha\gamma)^k \|x_0 - x^\star\|^2 +  \frac{\alpha \sigma^2}{2\gamma}.
    \label{eq:grad_err_stoch}
\end{equation*}
Notice that
\begin{equation*}
    \lim_{k \rightarrow \infty} \sup \mathbb{E}[\|x_{k} - x^\star\|^2 ] = \frac{\alpha \sigma^2}{2\gamma},
\end{equation*}
which completes the proof.
\end{proof}

\subsection{Main Results}
We can now begin to derive the convergence properties of the stochastic NEAR-DGD$^t$ method. We will closely follow the analysis in \cite{near_dgd_2017}. We start by proving that the magnitude of the system-wide stochastic NEAR-DGD$^t$ iterates is upper bounded in expectation.

\begin{lem}
\label{lem:bounded_iter}
\textbf{(Bounded iterates)} Let $\{\mathbf{x}_k\}$ and $\{\mathbf{y}_k\}$ be the sequences generated by the stochastic NEAR-DGD$^t$ method under Assumptions~\ref{assum_fi_mu_L} and \ref{assm:err_stoch}, from initial point $\mathbf{y}_0 \in \mathbb{R}^{np}$. Also, let the steplength $\alpha$ satisfy 
\begin{equation*}
    \alpha \leq \min_i \bigg(\frac{2}{\mu_i + L_i}\bigg).
\end{equation*}
Then the $\mathbf{x}_k$ and $\mathbf{y}_k$ iterates of the stochastic NEAR-DGD$^t$ method are bounded in expectation for all $k=0,1,2,...$, namely
\begin{equation*}
    \mathbb{E}[\|\mathbf{x}_k\|^2]\leq D^2\text{, }\mathbb{E}[\|\mathbf{y}_k\|^2] \leq D^2,
\end{equation*}
where $D^2=2\|\mathbf{y}_0-\mathbf{u}^\star\|^2 +  \frac{8+2\nu^3}{\nu^3}\|\mathbf{u}^\star\|^2+  \frac{2}{\nu^2}\Delta$, $\mathbf{u}^\star=[u_1^\star,...,u_n^\star]$, $u_i^\star=\arg \min_{u_i} f_i(u_i) $, $\nu = 2\alpha\gamma$, $\gamma = \min_i \gamma_i$, $\gamma_i = \frac{\mu_i L_i}{\mu_i + L_i}$ and $\Delta= n \alpha^2 \sigma^2$.
\end{lem}
\begin{proof} At each iteration of (\ref{eq:near_dgd_it_n_2}), agent $i \in \{1,...,n\}$ takes a stochastic gradient step on its local function $f_i$. Therefore, by Lemma~\ref{lem:grad_err_cen} the local iterates $y_{i,k}$ satisfy
\begin{equation*}
    \begin{split}
   \mathbb{E}[ \|y_{i,k+1} - &u_i^\star\|^2 | \mathcal{F}_k] =\\
   &= \mathbb{E}[ \|x_{i,k} -\alpha g_i(x_{i,k},\xi_{i,k}) - u_i^\star\|^2| \mathcal{F}_k] \\
    &\leq (1-2\alpha \gamma_i)\|x_{i,k} - u_i^\star\|^2 + \alpha^2\sigma^2,
    \end{split}
\end{equation*}
where $\gamma_i=\frac{\mu_i L_i}{\mu_i + L_i}$.

For the system $\mathbf{y}_k$ iterates, we therefore have
\begin{equation*}
    \begin{split}
       \mathbb{E}[ \|\mathbf{y}_{k+1} - &\mathbf{u}^\star\|^2 | \mathcal{F}_k] = \sum_{i=1}^n \mathbb{E} [\|y_{i,k+1}-u_i^\star\|^2 | \mathcal{F}_k] \\
        &\leq \sum_{i=1}^n \bigg((1-2\alpha \gamma_i)\|x_{i,k} - u_i^\star\|^2 + \alpha^2 \sigma^2\bigg) \\
        & \leq (1-2\alpha\gamma)\sum_{i=1}^n\|x_{i,k}-u_i^\star\|^2 +  n \alpha^2 \sigma^2 \\
        & = (1-\nu) \|\mathbf{x}_k - \mathbf{u}^\star\|^2 + \Delta,
    \end{split}
\end{equation*}
where we invoked Lemma~\ref{lem:grad_err_cen} for the first inequality. We obtain the second inequality from the definition of $\gamma$ and the final equality from the definitions of $\nu$ and $\Delta$.

By the definition of $\mathbf{x}_k$ (\ref{eq:near_dgd_it_n_1}), we further have
\begin{equation*}
    \begin{split}
      \mathbb{E}[ \|\mathbf{y}_{k+1} - &\mathbf{u}^\star\|^2 | \mathcal{F}_k]   \leq (1-\nu)\|\mathbf{x}_k - \mathbf{u}^\star\|^2 + \Delta \\
    & = (1-\nu)\|\mathbf{Z}^t(\mathbf{y}_k -\mathbf{u}^\star) - (I-\mathbf{Z}^t) \mathbf{u}^\star\|^2 + \Delta,
    \end{split}
\end{equation*}
where we have used (\ref{eq:near_dgd_it_n_1}) and added and subtracted $\mathbf{Z}^t\mathbf{u}^\star$. 

Notice that the following relation holds
\begin{equation*}
\begin{split}
    \|\mathbf{Z}^t(\mathbf{y}_k -\mathbf{u}^\star) - (I-\mathbf{Z}^t) &\mathbf{u}^\star\|^2 \leq (1+\nu) \|\mathbf{Z}^t(\mathbf{y}_k -\mathbf{u}^\star)\|^2\\
    &+ (1+\nu^{-1})\|(I-\mathbf{Z}^t) \mathbf{u}^\star\|^2,
    \end{split}
\end{equation*}
which yields
\begin{equation*}
    \begin{split}
    \mathbb{E}[ \|\mathbf{y}_{k+1} - &\mathbf{u}^\star\|^2 | \mathcal{F}_k]  \leq (1-\nu^2)\|\mathbf{Z}^t \mathbf{y}_k -\mathbf{u}^\star\|^2\\
        &+\frac{1-\nu^2}{\nu}\|(I-\mathbf{Z}^t )\mathbf{u}^\star\|^2 + \Delta \\
        & \leq (1-\nu^2)\|\mathbf{Z}^t \|^2\|\mathbf{y}_k -\mathbf{u}^\star\|^2\\
        &+\frac{1}{\nu}\|I-\mathbf{Z}^t\|^2\|\mathbf{u}^\star\|^2 + \Delta \\
        & \leq (1-\nu^2) \|\mathbf{y}_k -\mathbf{u}^\star\|^2+ \frac{4}{\nu}\| \mathbf{u}^\star\|^2 + \Delta,
    \end{split}
\end{equation*}
where we used Cauchy-Schwarz to get the second inequality and the spectral properties of $\mathbf{W}$ for the final inequality.

Recursively computing the expectation conditioned on the initial $\sigma$-algebra $\mathcal{F}_0$, i.e. the full expectation, yields
\begin{equation*}
    \begin{split}
        \mathbb{E}[\|\mathbf{y}_{k+1}-\mathbf{u}^\star&\|^2] \leq (1-\nu^2) ^{k+1}\|\mathbf{y}_0-\mathbf{u}^\star\|^2 \\
        &+ \frac{4}{\nu}\|\mathbf{u}^\star\|^2\sum_{m=1}^{k+1} (1-\nu^2)^m + \Delta\sum_{m=0}^{k}( 1-\nu^2)^m \\
         & \leq \|\mathbf{y}_0-\mathbf{u}^\star\|^2+ \frac{4(1-\nu^2)}{\nu^3}\|\mathbf{u}^\star\|^2+  \frac{\Delta}{\nu^2} \\
         & \leq \|\mathbf{y}_0-\mathbf{u}^\star\|^2 +  \frac{4}{\nu^3}\|\mathbf{u}^\star\|^2+  \frac{\Delta}{\nu^2} .
    \end{split}
\end{equation*}
Note that for all $k=0,1,2...$, we have
\begin{equation*}
    \begin{split}
        \|\mathbf{y}_{k+1}\|^2 &\leq 2\|\mathbf{y}_{k+1}-\mathbf{u}^\star\|^2 + 2\|\mathbf{u}^\star\|^2.
    \end{split}
    \label{eq:bound_y}
\end{equation*}
Taking the full expectation on both sides, we obtain
\begin{equation*}
    \begin{split}
         \mathbb{E}[\|\mathbf{y}_{k+1}\|^2] &\leq  2\mathbb{E}[\|\mathbf{y}_{k+1}-\mathbf{u}^\star\|^2] + 2\|\mathbf{u}^\star\|^2 \\
         & \leq 2 \|\mathbf{y}_0-\mathbf{u}^\star\|^2 +  \frac{8+2\nu^3}{\nu^3}\|\mathbf{u}^\star\|^2+  \frac{2}{\nu^2}\Delta.
    \end{split}
    \label{eq:bound_y}
\end{equation*}
Finally, for the $\mathbf{x}_k$ iterates we have
\begin{equation*}
    \begin{split}
       \mathbb{E}[ \|\mathbf{x}_{k}\|^2] &= \mathbb{E}[\|\mathbf{Z}^t\mathbf{y}_{k}\|^2] \\
        & \leq \|\mathbf{Z}^t\|^2\mathbb{E}[\|\mathbf{y}_{k}\|^2] \\
        & \leq \mathbb{E}[\|\mathbf{y}_{k}\|^2].
    \end{split}
    \label{eq:bound_x}
\end{equation*}
where we have used (\ref{eq:near_dgd_it_n_1}), the Cauchy-Schwarz inequality and the fact that $\|\mathbf{Z}^t\|=1$.
Applying the definition of $D^2$ completes the proof.
\end{proof}

In the following Corollary, we derive the relations between the average stochastic NEAR-DGD$^t$ iterates $\bar{x}_k=\mathcal{M}_n \mathbf{x}_k$ and \mbox{$\bar{y}_k=\mathcal{M}_n \mathbf{y}_k$}.
\begin{cor}
\label{cor:mean_iter}
\textbf{(Average iterates)} Let $\bar{x}_k=\mathcal{M}_n \mathbf{x}_k$ and \mbox{$\bar{y}_k=\mathcal{M}_n \mathbf{y}_k$} denote the average iterates generated by the stochastic NEAR-DGD$^t$ method. Then the following relations hold
\begin{gather*}
    \bar{x}_k = \bar{y}_k,\\
    \bar{y}_{k+1} = \bar{x}_k - \alpha \bar{g}_k
\end{gather*}
where $\bar{g}_k = \mathcal{M}_n \mathbf{g}(\mathbf{x}_k,\pmb{\xi}_k)$.
\end{cor}
\begin{proof}
This result is obtained in a straightforward manner by multiplying (\ref{eq:near_dgd_it_n_1}) and (\ref{eq:near_dgd_it_n_2}) with $\mathcal{M}_n$ and noticing that $\mathcal{M}_n \mathbf{Z} = \mathcal{M}_n$ by the double stochasticity of $\mathbf{Z}$.
\end{proof}

We proceed by bounding the  variance of the local stochastic NEAR-DGD$^t$ iterates $x_{i,k}$ and $y_{i,k}$. The variance of the local iterates is a measure of the distance to consensus and should ideally approach zero.
\begin{lem}
\label{lem:bound_dev}
\textbf{(Bounded variance)} Let $h_k$ denote the average of all local gradients at iteration $k$ and $\bar{h}_k$ the gradient of $\bar{f}$ at $\bar{x}_k$, i.e.
\begin{equation*}
    h_k= \frac{1}{n}\sum_{i=1}^n \nabla f_i(x_{i,k})\text{, } \bar{h}_k= \frac{1}{n}\sum_{i=1}^n \nabla f_i(\bar{x}_k).
\end{equation*}
Also let $x_{i,k}$ and $y_{i,k}$ be the local iterates produced by the stochastic NEAR-DGD$^t$ method under Assumptions~\ref{assum_fi_mu_L} and \ref{assm:err_stoch} and with steplength $\alpha$ satisfying
\begin{equation*}
    \alpha \leq \min_i \bigg ( \frac{2}{\mu_i + L_i} \bigg ).
\end{equation*}
Then the following bounds hold for all $i=1,...,n$ and \mbox{$k=0,1,2,...$}
\begin{gather*}
     \mathbb{E}[\|x_{i,k}-\bar{x}_k\|^2] \leq \beta^{2t} D^2,\\
     \mathbb{E}[\|h_k - \bar{h}_k\|^2]\leq \beta^{2t} L^2 D^2, \\
     \mathbb{E}[\|y_{i,k}-\bar{y}_k\|^2]\leq 2\beta^{2t} D^2 + 8D^2,
\end{gather*}
where $\beta$ is the second largest singular value of $\mathbf{W}$.
\end{lem}
\begin{proof}
Observing that $\bar{x}_k=\bar{y}_k$ from Corollary~\ref{cor:mean_iter}, we obtain
\begin{equation*}
\begin{split}
    \|x_{i,k}-\bar{x}_k\|^2 & = \|x_{i,k}-\bar{y}_k\|^2 \\
    &\leq \|\mathbf{x}_k-\mathcal{M}_n\mathbf{y}_k\|^2 \\
    & = \|\mathbf{Z}^t \mathbf{y}_k -\mathcal{M}_n\mathbf{y}_k\|^2 \\
    &\leq \|\mathbf{Z}^t - \mathcal{M}_n\|^2\|\mathbf{y}_k\|^2\\
    &\leq \beta^{2t} \|\mathbf{y}_k\|^2,
\end{split}
\end{equation*}

We furthermore have
\begin{equation*}
    \begin{split}
        \|h_k-\bar{h}_k\|^2 &=\frac{1}{n^2} \bigg\|\sum_{i=1}^n \big ( \nabla f_i(x_{i,k})-\nabla f_i(\bar{x}_k) \big)\bigg\|^2\\
        &\leq \frac{1}{n} \sum_{i=1}^n \|\nabla f_i(x_{i,k}) - \nabla f_i(\bar{x}_k)\|^2 \\
        &\leq \frac{1}{n} \sum_{i=1}^n L_i^2 \|x_{i,k} - \bar{x}_k\|^2 \\
        &\leq \beta^{2t} L^2 \|\mathbf{y}_k\|^2,
    \end{split}
\end{equation*}
where $L = \max_i L_i$. We obtain the first inequality from Cauchy-Schwarz, the second inequality from Assumption~\ref{assum_fi_mu_L} and the last inequality from the immediately previous result.

Finally, we bound the variance of the local $y_{i,k}$ iterates as
\begin{equation*}
    \begin{split}
        \|y_{i,k}-\bar{y}_k\|^2 & \leq 2\|x_{i,k} - \bar{y}_k\|^2 + 2\|y_{i,k}-x_{i,k}\|^2 \\
        &\leq 2\beta^{2t} \|\mathbf{y}_k\|^2 + 2\|\mathbf{y}_k-\mathbf{x}_k\|^2 \\
        &= 2\beta^{2t} \|\mathbf{y}_k\|^2 + 2\|\mathbf{y}_k-\mathbf{Z}^t\mathbf{y}_k\|^2 \\
        &\leq 2\beta^{2t} \|\mathbf{y}_k\|^2+ 2\|I-\mathbf{Z}^t\|^2\|\mathbf{y}_k\|^2 \\
        &\leq 2\beta^{2t} \|\mathbf{y}_k\|^2 + 8\|\mathbf{y}_k\|^2.
    \end{split}
\end{equation*}
where the equality is due to (\ref{eq:near_dgd_it_n_1}) and the last inequality due to the spectral properties of $\mathbf{W}$. 

Taking the full expectation on both sides of all previous results and applying Lemma~\ref{lem:bounded_iter} concludes the proof. 
\end{proof}

We proceed by showing that the average stochastic gradient $\bar{g}_k$ is an unbiased estimator of the true local gradient average $h_k$ with bounded variance.
\begin{lem}
\label{lem:unbiased_gk}
\textbf{(Average stochastic gradient)}
Let $\bar{g}_k = \mathcal{M}_n \mathbf{g}(\mathbf{x}_k,\pmb{\xi}_k)$ and $h_k=\frac{1}{n}\sum_{i=1}^n \nabla f_i(x_{i,k})$. Then if Assumption~\ref{assm:err_stoch} holds, $\bar{g}_k$ is an unbiased estimator of $h_k$ with bounded variance, i.e.
\begin{gather*}
    \mathbb{E}[\bar{g}_k | \mathcal{F}_k] = h_k\text{, }
    \mathbb{E}[\|\bar{g}_k - h_k \|^2| \mathcal{F}_k] \leq \frac{\sigma^2}{n}.
\end{gather*}
\end{lem}
\begin{proof}
From the definitions of $\bar{g}_k$ and $h_k$ and Assumption~\ref{assm:err_stoch} we have
\begin{equation*}
\begin{split}
    \mathbb{E}[\bar{g}_k| \mathcal{F}_k] = \frac{1}{n} \sum_{i=1}^n \mathbb{E}[ g_i(x_{i,k},\xi_{i,k}) | x_{i,k} ] =h_k,
    \end{split}
\end{equation*}
and
\begin{equation*}
\begin{split}
     \mathbb{E}[&\|\bar{g}_k - h_k\|^2 | \mathcal{F}_k]=\\&= \frac{1}{n^2} \sum_{i=1}^n \mathbb{E}[\| g_i(x_{i,k},\xi_{i,k}) - \nabla f_i (x_{i,k}) \|^2 | x_{i,k} ]\leq \frac{\sigma^2}{n},
    \end{split}
\end{equation*}
which proves the desired result. 
\end{proof}
Notice that variance bound of $\bar{g}_k$ is scaled by the total number of agents $n$, which is equivalent to centralized mini-batching with batch size $n$.

We are now ready to prove that the average iterates $\bar{x}_k$ produced by the stochastic NEAR-DGD$^t$ method converge.
\begin{thm}
\label{thm:mean_conv}
\textbf{(Distance to minimum)} 
Let $\bar{x}_k$ be the average iterates generated by the stochastic NEAR-DGD$^t$ method under Assumptions~\ref{assum_fi_mu_L} and \ref{assm:err_stoch} and let the steplength $\alpha$ satisfy 
\begin{equation*}
    \alpha \leq \min_i \bigg ( \frac{2}{\mu_i+L_i} \bigg ).
\end{equation*}
Then the distance of $\bar{x}_k$ to the solution is bounded in expectation for all $k=0,1,2,...$, namely
\begin{equation*}
    \begin{split}
        \mathbb{E}[\|\bar{x}_{k}-x^\star\|^2] &\leq c_1^k \|\bar{x}_0 - x^\star\|^2+ \frac{c_2^2 \beta^{2t}}{1-c_1} + \frac{\alpha^2 \sigma^2}{n(1-c_1)},
    \end{split}
\end{equation*}
where $c_1=(1+\psi)(1-2\alpha \gamma_{\bar{f}})$, $\psi < \frac{2\alpha\gamma_{\bar{f}}}{1-2\alpha\gamma_{\bar{f}}}$ a positive constant, $\gamma_{\bar{f}} = \frac{\mu_{\bar{f}} L_{\bar{f}}}{\mu_{\bar{f}}+L_{\bar{f}}}$ and $c_2^2=\alpha^2(1+\psi^{-1})L^2D^2$.
\end{thm}

\begin{proof}
By Corollary~\ref{cor:mean_iter}, the mean iterates $\bar{x}_k$ satisfy
\begin{equation*}
    \bar{x}_{k+1} = \bar{x}_k- \alpha \bar{g}_k.
\end{equation*}
Therefore, for the sequence $\bar{x}_k$, we have
\begin{equation*}
    \begin{split}
        \|\bar{x}_{k+1}-x^\star\|^2 &= \|\bar{x}_k - \alpha \bar{g}_k - x^\star\|^2 \\
        & = \|\bar{x}_k - x^\star\|^2 - 2 \alpha \langle \bar{x}_k - x^\star,\bar{g}_k \rangle + \alpha^2 \|\bar{g}_k\|^2.
    \end{split}
\end{equation*}
Taking the conditional expectation with respect to $\mathcal{F}_k$ and applying Lemma~\ref{lem:unbiased_gk}, yields
\begin{equation*}
    \begin{split}
        \mathbb{E}[\|\bar{x}_{k+1}-x^\star\|^2 | \mathcal{F}_k] &= \|\bar{x}_k - x^\star\|^2 - 2 \alpha \langle \bar{x}_k - x^\star,h_k \rangle \\&+ \alpha^2 \mathbb{E}[\|\bar{g}_k\|^2| \mathcal{F}_k]\\
        & \leq \|\bar{x}_k - x^\star\|^2 - 2 \alpha \langle \bar{x}_k - x^\star,h_k \rangle \\&+ \alpha^2 \|h_k\|^2 + \frac{\alpha^2 \sigma^2}{n}\\
        &= \|\bar{x}_k - \alpha h_k - x^\star\|^2 + \frac{\alpha^2 \sigma^2}{n}.
    \end{split}
\end{equation*}
For the first term in the right-hand side of the previous inequality and some positive constant $\psi < \frac{2 \alpha \gamma_{\bar{f}}}{1-2\alpha\gamma_{\bar{f}}}$, we have
\begin{equation*}
    \begin{split}
     \|\bar{x}_k - \alpha h_k - x^\star \|^2 &=\|\bar{x}_k - \alpha \bar{h}_k - x^\star + \alpha (\bar{h}_k - h_k)\|^2 \\
       &\leq (1+\psi) \|\bar{x}_k - \alpha \bar{h}_k - x^\star\|^2 \\
        &+ \alpha^2(1+\psi^{-1})\|\bar{h}_k - h_k\|^2.
    \end{split}
\end{equation*}
Consider now the following optimization problem
\begin{equation}
    \min_{x \in \mathbb{R}^p} \bar{f}(x) = \frac{1}{n} \sum_{i=1}^n f_i(x).
    \label{eq:prob_avg}
\end{equation}
Notice that $\nabla \bar{f}(\bar{x}_k) = \bar{h}_k$; therefore, the quantity $\bar{x}_k - \alpha \bar{h}_k$ can be interpreted as one step of an exact gradient method for solving problem (\ref{eq:prob_avg}). Also notice that the solution for this problem is $x^\star$ and that Lemma~\ref{lem:global_mu_L} applies to $\bar{f}$. 

Given that the steplength satisfies $\alpha \leq \min_i \bigg (\frac{2}{\mu_i + L_i} \bigg ) \leq \frac{2}{\mu_{\bar{f}} + L_{\bar{f}}}$, by \cite[Theorem 2.1.14, Chapter 2]{nesterov_introductory_1998} we have
\begin{equation*}
    \begin{split}
        \|\bar{x}_{k}-\alpha \bar{h}_k -x^\star\|^2 \leq (1-2 \alpha \gamma_{\bar{f}})\|\bar{x}_k-x^\star\|^2,
    \end{split}
\end{equation*}
where $\gamma_{\bar{f}} = \frac{\mu_{\bar{f}} L_{\bar{f}}}{\mu_{\bar{f}}+L_{\bar{f}}}$.

Combining all of the above, we get
\begin{equation*}
    \begin{split}
        \mathbb{E}[\|\bar{x}_{k+1}-x^\star\|^2 | \mathcal{F}_k] &\leq (1+\psi)(1-2 \alpha \gamma_{\bar{f}}) \|\bar{x}_k - x^\star\|^2\\&+ \alpha^2(1+\psi^{-1})\|\bar{h}_k-h_k\|^2 + \frac{\alpha^2 \sigma^2}{n}.
    \end{split}
\end{equation*}
Finally, by taking the full expectation and applying Lemma~\ref{lem:bound_dev} and the definitions of $c_1$ and $c_2$, we obtain
\begin{equation*}
    \begin{split}
        \mathbb{E}[\|\bar{x}_{k+1}-x^\star\|^2] &\leq c_1 \mathbb{E}[\|\bar{x}_k - x^\star\|^2]+ c_2^2 \beta^{2t} + \frac{\alpha^2 \sigma^2}{n},
    \end{split}
\end{equation*}
or by induction
\begin{equation}
\label{eq:theorem_barx}
    \begin{split}
        \mathbb{E}[\|\bar{x}_{k}-x^\star\|^2] &\leq c_1^k \|\bar{x}_0 - x^\star\|^2 + \frac{c_2^2 \beta^{2t}}{1-c_1} + \frac{\alpha^2 \sigma^2}{n(1-c_1)},
    \end{split}
\end{equation}
which completes the proof.
\end{proof}
The second term in the right-hand side of (\ref{eq:theorem_barx}) can be interpreted as the error due to network connectivity and the third term as the error due to stochastic noise in the local gradients. The variable $\psi$ can take any value in the interval $\big(0,2\alpha\gamma_{\bar{f}}/(1-2\alpha\gamma_{\bar{f}})\big)$ and reflects a trade-off between the bounds on convergence accuracy and speed. Notice that as $\psi$ diminishes, $\lim_{\psi \rightarrow 0} c_1 = 1-2\alpha\gamma_{\bar{f}}$ and we approach the convergence rate of centralized SGD (see Lemma~\ref{lem:grad_err_cen}). However, we also have that $\lim_{\psi \rightarrow 0} c_2 = \infty$ and the network-related error term in (\ref{eq:theorem_barx}) becomes arbitrarily large. 

Lemma~\ref{lem:bound_dev} states that the distance of the local $x_{i,k}$ and $y_{i,k}$ iterates to $\bar{x}_k$ is bounded. Thus, the convergence of $\bar{x}_k$ implies that the local iterates also converge.
\begin{cor}
\label{cor:local_conv}
\textbf{(Convergence of local iterates)} Let $x_{i,k}$ and $y_{i,k}$ be the local iterates generated by the stochastic NEAR-DGD$^t$ method under  Assumptions~\ref{assum_fi_mu_L} and \ref{assm:err_stoch} with a steplength $\alpha$ satisfying
\begin{equation*}
    \alpha \leq \min_i \bigg ( \frac{2}{\mu_i+L_i} \bigg ).
\end{equation*}  
Then the distance of $x_{i,k}$ and $y_{i,k}$ to the solution is bounded for all $i=1,...,n$ and $k=0,1,2,...$
\begin{equation*}
    \begin{split}
        \mathbb{E}[\|x_{i,k}-x^\star\|^2] &\leq c_1^k \cdot 2\|\bar{x}_{0}-x^\star\|^2 + \frac{2 c_2^2 \beta^{2t}}{1-c_1} \\&+ \frac{2\alpha^2\sigma^2}{n(1-c_1)} + 2\beta^{2t}D^2.
    \end{split}
\end{equation*}
and
\begin{equation*}
\begin{split}
\mathbb{E}[\|y_{i,k}-x^\star\|^2] &\leq c_1^k \cdot 2\|\bar{x}_{0}-x^\star\|^2 + \frac{2 c_2^2 \beta^{2t}}{1-c_1} \\&+ \frac{2\alpha^2\sigma^2}{n(1-c_1)} + 4\beta^{2t}D^2 + 16D^2.
\end{split}
\end{equation*}
\end{cor}
\begin{proof}
Notice that for the local $x_{i,k}$ iterates we have
\begin{equation*}
\begin{split}
     \|x_{i,k}-x^\star\|^2 & \leq \|x_{i,k} - \bar{x}_k + \bar{x}_k - x^\star\|^2\\
     &\leq 2\|\bar{x}_k - x^\star\|^2+2\|x_{i,k} - \bar{x}_k\|^2.
\end{split}
\end{equation*}
Similarly, for the local $y_{i,k}$ iterates we have
\begin{equation*}
\begin{split}
     \|y_{i,k}-x^\star\|^2 &= \|y_{i,k} - \bar{y}_k + \bar{y}_k - x^\star\|^2\\
     &\leq 2\|y_{i,k} - \bar{y}_k\|^2 + 2\|\bar{x}_k - x^\star\|^2
\end{split}
\end{equation*}
where in the last inequality we used the fact that $\bar{y}_k=\bar{x}_k$ from Corollary~\ref{cor:mean_iter}. Calculating the full expectation on both sides of each inequality and applying Lemma~\ref{lem:bound_dev} and Theorem~\ref{thm:mean_conv} completes the proof. 
\end{proof}

For the remaining two theorems of this section, consider the variant of the NEAR-DGD method where we increase the number of consensus steps by one at every iteration, i.e. $t(k)=k$. We will refer to this variant as NEAR-DGD$^+$.

\begin{thm}\textbf{(Convergence neighborhood of stochastic NEAR-DGD$^{+}$)} 
\label{thm:ndgd_p_error}
Let $\bar{x}_k$ be the average iterates produced by the stochastic NEAR-DGD$^{+}$ method with steplength $\alpha \leq \min_i \big( \frac{2}{\mu_i + L_i}\big)$ and let Assumptions~\ref{assum_fi_mu_L} and \ref{assm:err_stoch} hold. Then $\bar{x}_k$ converges to a neighborhood of the optimal solution with size $\mathcal{O}\big(\frac{\alpha^2\sigma^2}{n(1-c_1)}\big)$ where $c_1=(1+\psi)(1-2\alpha \gamma_{\bar{f}})$, $\psi < 2\alpha\gamma_{\bar{f}}/(1-2\alpha\gamma_{\bar{f}})$ a positive constant$, \gamma_{\bar{f}}=\frac{\mu_{\bar{f}} L_{\bar{f}}}{\mu_{\bar{f}} + L_{\bar{f}}}$ and $L_{\bar{f}},{\mu_{\bar{f}}}$ are defined in Lemma~\ref{lem:global_mu_L}.
\end{thm}
\begin{proof}
In the case of the stochastic NEAR-DGD$^{+}$ method, the result of Theorem~\ref{thm:mean_conv} transforms to 
\begin{equation*}
    \begin{split}
        \mathbb{E}[\|\bar{x}_{k}-x^\star\|^2] &\leq c_1^k \|\bar{x}_0 - x^\star\|^2 + \frac{c_2^2 \beta^{2k}}{1-c_1} + \frac{\alpha^2 \sigma^2}{n(1-c_1)},
    \end{split}
\end{equation*}
As the number of iterations increases, $\lim_{k \rightarrow \infty} \beta^{2k} = 0$. We therefore have
\begin{equation*}
    \lim_{k \rightarrow \infty} \sup  \mathbb{E}[\|\bar{x}_{k}-x^\star\|^2] =  \frac{\alpha^2 \sigma^2}{n(1-c_1)},
\end{equation*} 
which proves the desired result.
\end{proof}
Notice that when $\psi$ approaches zero, $\lim_{\psi \rightarrow 0} 1-c_1 = 2 \alpha \gamma_{\bar{f}}$ and we approach the error neighborhood of centralized mini-batching with $n$ samples as per Lemma~\ref{lem:grad_err_cen}.

\begin{thm}\textbf{(Linear convergence of stochastic NEAR-DGD$^{+}$)}
\label{thm:ndgd_p_rate}
Let $\bar{x}_k$ be the average iterates produced by the stochastic NEAR-DGD$^{+}$ method with steplength $\alpha \leq \min_i \big( \frac{2}{\mu_i + L_i}\big)$ under Assumptions~\ref{assum_fi_mu_L} and \ref{assm:err_stoch}. Then $\bar{x}_k$ converges in expectation to a neighborhood of the solution with linear rate
\begin{equation*}
    \begin{split}
        \mathbb{E}[\|\bar{x}_{k}-x^\star\|^2] &\leq C \theta^k + \frac{\alpha^2\sigma^2}{n(1-c_1)},
    \end{split}
\end{equation*}
where 
\begin{equation*}
\begin{split}
    C = \max \bigg \{ \|\bar{x}_0-x^\star\|^2,\frac{2c_2^2}{1-c_1}\bigg \},\theta = \max \bigg\{\beta^2,\frac{c_1+1}{2}\bigg \}.
    \end{split}
\end{equation*}
\end{thm}
\begin{proof}
We will prove this theorem by induction. The result holds trivially for $k=0$ and let it also hold at iteration $k$. Then at iteration $k+1$, by Theorem~\ref{thm:mean_conv} we have
\begin{equation*}
\label{eq:conv_rate}
    \begin{split}
        \mathbb{E}[\|\bar{x}_{k+1}-x^\star\|] &\leq c_1 \mathbb{E}[\|\bar{x}_k - x^\star\|^2]+ c_2^2 \beta^{2k} + \frac{\alpha^2 \sigma^2}{n},\\
        &\leq c_1\bigg (C \theta^k + \frac{\alpha^2\sigma^2}{n(1-c_1)}\bigg)  +  c_2^2\beta^{2k} + \frac{\alpha^2\sigma^2}{n}\\
        &= C \theta^k \bigg (c_1 + \frac{c_2^2 \beta^{2k}}{C \theta^k} \bigg) + \frac{\alpha^2\sigma^2}{n(1-c_1)},
    \end{split}
\end{equation*}
where we used the assumption that the result holds at iteration $k$. 

We furthermore have
\begin{equation*}
    \begin{split}
        \mathbb{E}[\|\bar{x}_{k+1}-x^\star\|] &\leq C \theta^k \bigg(c_1 + \frac{c_2^2}{C} \bigg )+ \frac{\alpha^2\sigma^2}{n(1-c_1)} \\
        &\leq C \theta^k \bigg (c_1 + \frac{1-c_1}{2} \bigg ) + \frac{\alpha^2\sigma^2}{n(1-c_1)} \\
        &\leq C \theta^{k+1} + \frac{\alpha^2\sigma^2}{n(1-c_1)}, 
    \end{split}
\end{equation*}
where the first inequality is derived from the definition of $\theta$, the second inequality from the definition of $C$ and the third inequality from the definition of $\theta$. 

We have therefore demonstrated that the result holds at iteration $k+1$, which concludes the proof. 
\end{proof}
Theorems \ref{thm:ndgd_p_error} and \ref{thm:ndgd_p_rate} indicate that it is necessary to increase the number of consensus steps per iteration in order to suppress network-related error and achieve comparable performance to centralized mini-batching with the stochastic NEAR-DGD method. 

\section{Numerical Results}
\label{sec:res}
Consider the following logistic regression problem for binary classification
\begin{equation*}
 \min_{x \in \mathbb{R}^p} f(x) = \frac{1}{M}\sum_{s=1}^M \log(1+e^{-b_s \langle A_s , x \rangle}) + \frac{1}{M}\|x\|^2_2,
\end{equation*}
where $M$ is the total number of samples, $A \in \mathbb{R}^{M \times p}$ a feature matrix, $p$ the problem dimension and $b \in \{-1,1\}^{M}$ a vector of labels. We can solve a scaled version of this problem in a decentralized fashion by evenly distributing the samples among $n$ nodes and setting
\begin{equation*}
    f_i(x) = \frac{1}{|S_i|}\sum_{s \in S_i} \log(1+e^{-b_s \langle A_s , x \rangle}) + \frac{1}{M}\|x\|^2_2
\end{equation*}
where $S_i$ is the set of sample indices assigned to node $i$.

We conducted a numerical experiment using the mushrooms dataset \mbox{($p=118$, $M=8120$)} \cite{mushrooms} and a random network of $n=10$ nodes generated with the Erd\H{o}s-R\'{e}nyi model with edge probability $0.5$. At every iteration, nodes randomly draw with replacement $B=16$ samples from their local distributions and compute a mini-batch gradient. We tested several variants of the stochastic NEAR-DGD method against the stochastic versions of DGD \cite{NedicSubgradientConsensus,sundharram_distributed_2010} and EXTRA \cite{extra} and DSGT \cite{pu_distributed_2018-1}. The variants of NEAR-DGD are described using the following convention: $(a,-,-)$ signifies performing $a$ consensus steps at every iteration, while NEAR-DGD $(a,b,\times 2)$ denotes starting with $a$ consensus steps per iteration and doubling them every $b$ iterations. All methods shared the same steplength ($\alpha=1$) and drew the same samples at each iteration.

The results of a typical experiment run are presented in Figure~\ref{fig:results}. In the top left position, we have plotted the squared error $\|\bar{x}_k-x^\star\|^2$ against the number of iterations for all methods and for centralized mini-batching with $16n$ samples. All methods except DGD achieve almost identical performance to centralized mini-batching. The average squared error $\|\bar{x}_k-x^\star\|^2$ of the last $2000$ iterations against the total number of consensus steps for each method is shown in the top right corner of Figure~\ref{fig:results}. NEAR-DGD $(1,-,-)$ performs slightly worse than EXTRA and DSGT in terms of accuracy, while the remaining variants of NEAR-DGD reach the same accuracy as EXTRA and DSGT at the cost of more communication rounds. Note, however, that NEAR-DGD does not store previous gradients or more than one previous iterates. In the bottom left corner of Figure~\ref{fig:results}, we show the normalized standard deviation at the final iteration $\sqrt{\frac{1}{n}\sum_{i=1}^n\|x_{i,N}-\bar{x}_N\|^2}/\|\bar{x}_N\|$, for $N=25000$ total iterations. We observe that performing multiple communication rounds improves consensus among agents, especially if their number is increased gradually. Finally, in the bottom right position, we experiment with a less well-connected topology of path graph and local batch size $B=1$. We plot the squared error $\|\bar{x}_k-x^\star\|^2$ against the number of iterations. It can be seen that NEAR-DGD is the only method that approaches the performance of centralized mini-batching. This trend was consistent through all different runs of the experiment with the combination of path graph and $B=1$, and implies NEAR-DGD might be preferable in extreme cases where the network is poorly connected and the variance of the local stochastic gradients is high.

\begin{figure}
    \centering
    \includegraphics[width=.22\textwidth]{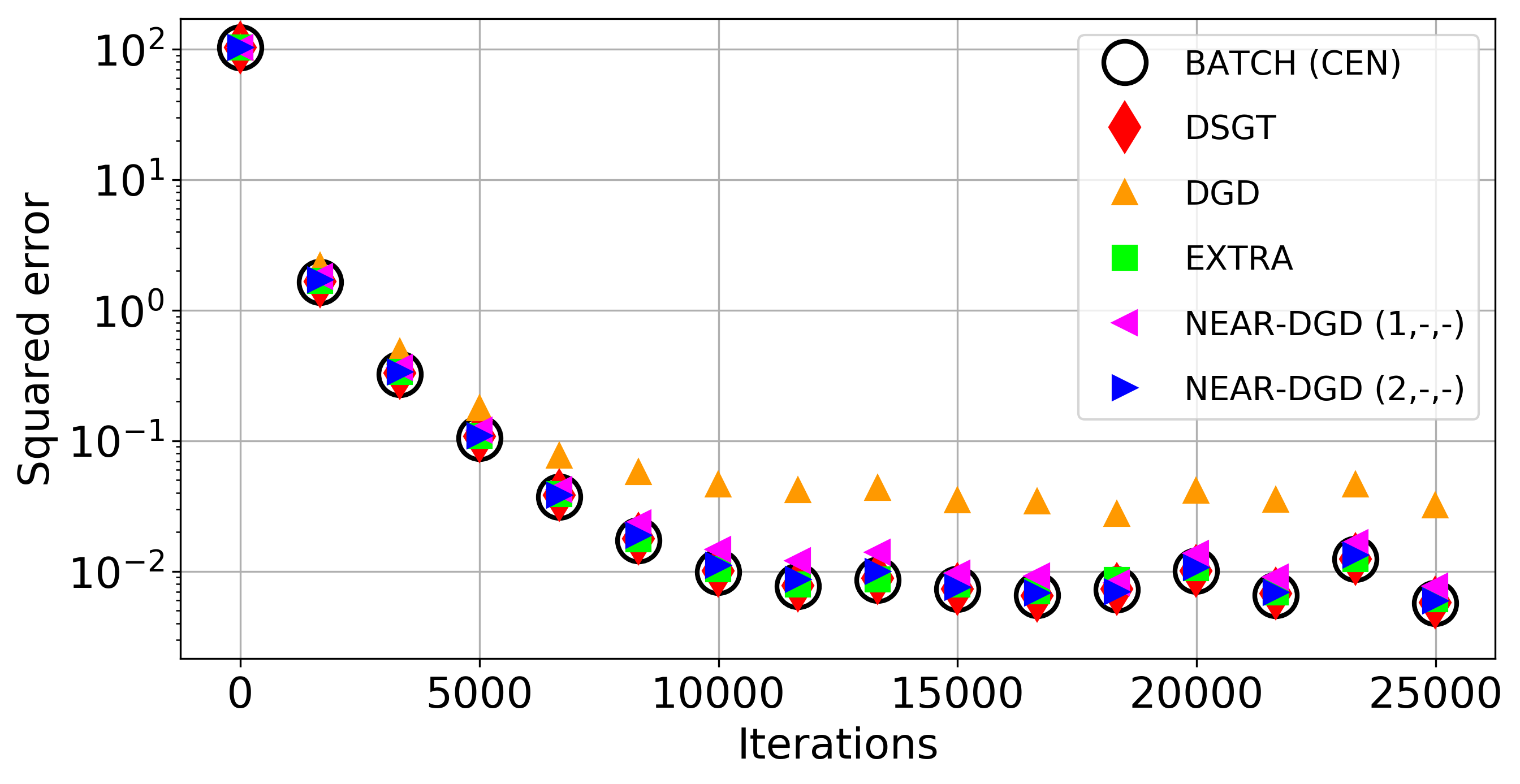}
    \includegraphics[width=.22\textwidth]{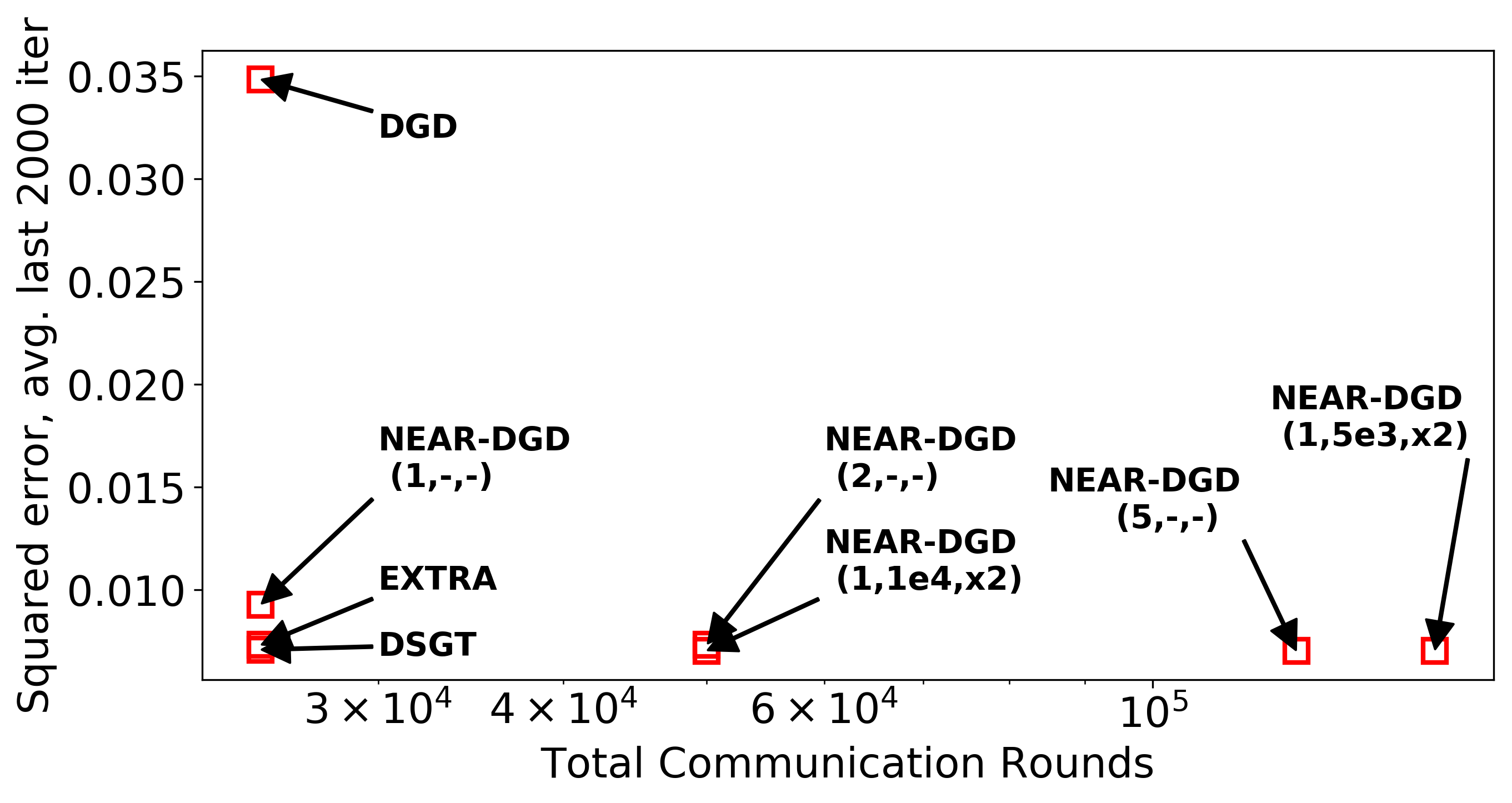}
    
    \includegraphics[width=.22\textwidth]{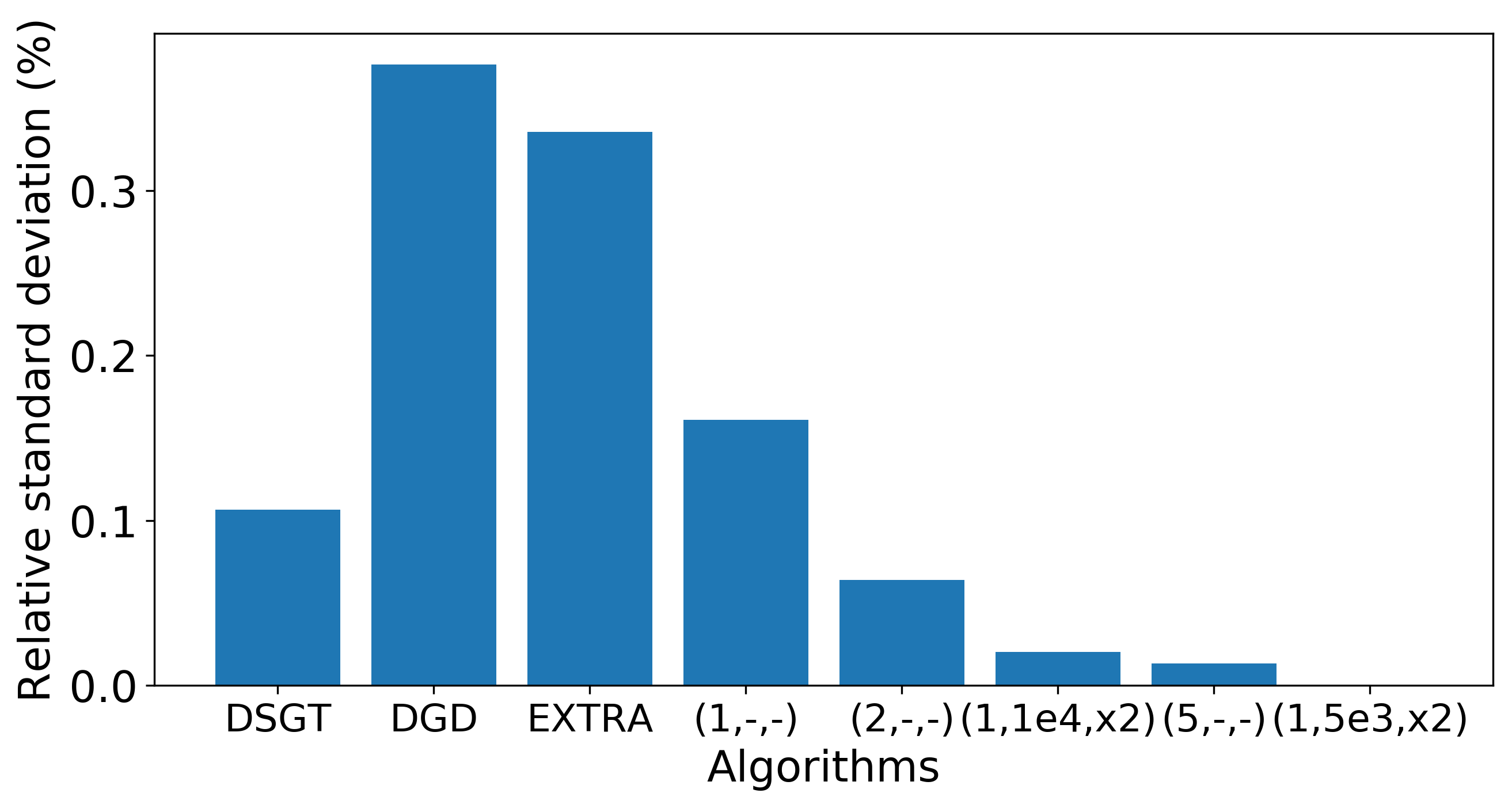}
    \includegraphics[width=.22\textwidth]{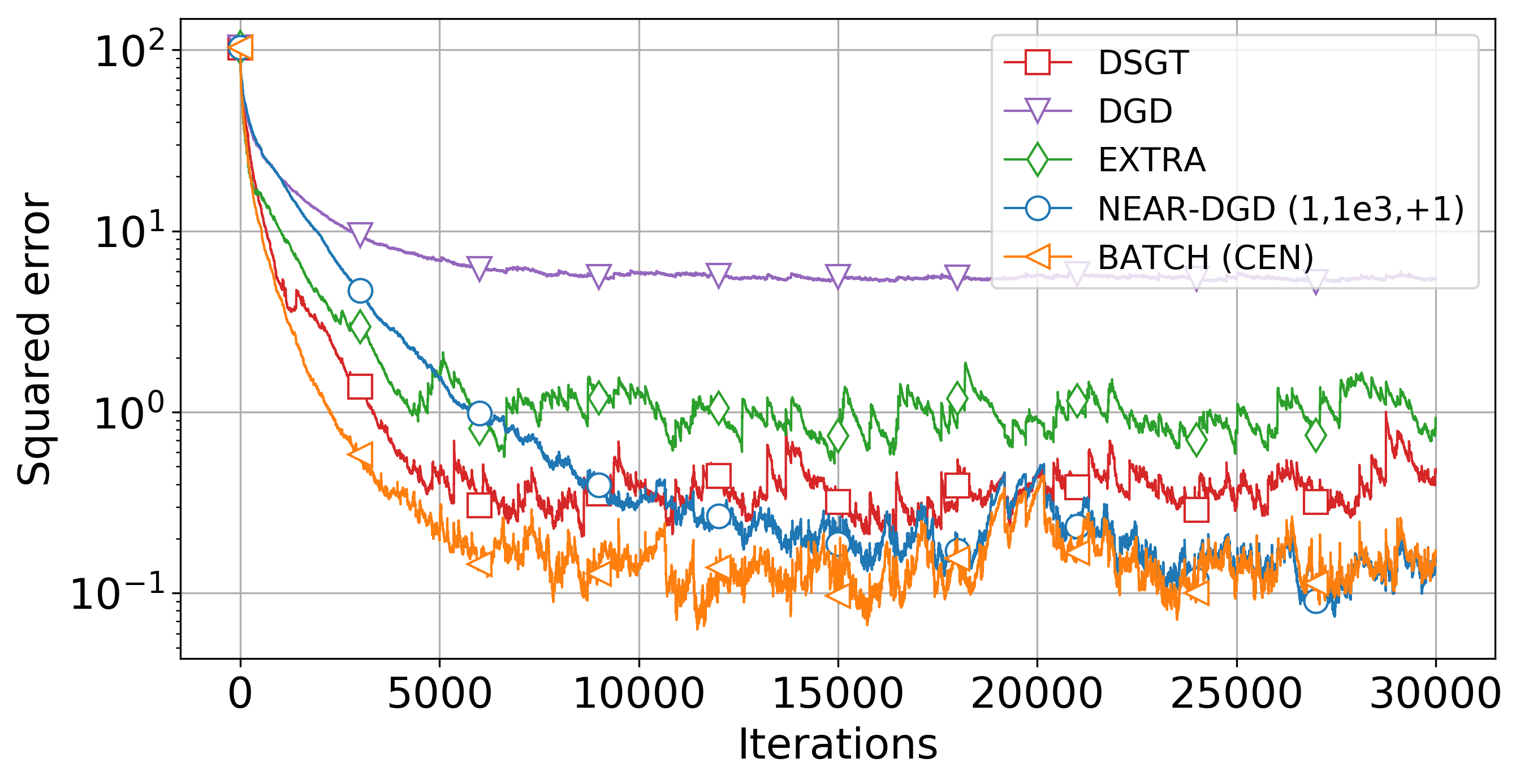}
    \caption{Squared error $\|\bar{x}_k-x^\star\|^2$ per iteration (top left), average squared error $\|\bar{x}_k-x^\star\|^2$ of the last $2000$ iterations plotted against the total number of communication rounds (top right), relative standard deviation at the final iteration $\sqrt{\frac{1}{n}\sum_{i=1}^n\|x_{i,N}-\bar{x}_N\|^2}/\|\bar{x}_N\|$ for $N=25000$ total iterations (bottom left) and squared error $\|\bar{x}_k-x^\star\|^2$ per iteration for the special case of path graph network topology and batch size $B=1$ (bottom right).}
    \label{fig:results}
\end{figure}

\section{Conclusion}
\label{sec:final}

In this paper, we analyzed the performance of Nested Distributed Gradient Method (NEAR-DGD) and its variants under the assumption of unbiased stochastic gradients with bounded variance. The strength of this method lies in its flexible framework, that alternates between gradient steps and a varying number of nested consensus steps which can be tuned depending on application-specific costs. Moreover, NEAR-DGD requires minimal storage.

Our analysis indicates that under the assumptions listed and for carefully chosen steplengths, a variant of the stochastic NEAR-DGD method converges in expectation to a neighborhood of the solution with linear rate for strongly convex functions. In addition, our method achieves a variance reduction effect similar to mini-batching, a trait that it shares with a number of other stochastic distributed algorithms. 

Finally, our numerical results show that our method is able to achieve comparable accuracy and convergence rates to other state-of-the-art algorithms. In addition to that, it accomplishes a stronger consensus between agents as opposed to other methods as a result of performing multiple communication rounds per iteration.









\bibliographystyle{ieeetr}
\bibliography{ref}

\begin{thebibliography}{10}

\bibitem{sensor1}
Q.~{Ling} and Z.~{Tian}, ``Decentralized sparse signal recovery for compressive
  sleeping wireless sensor networks,'' {\em IEEE Transactions on Signal
  Processing}, vol.~58, pp.~3816--3827, July 2010.

\bibitem{sensor2}
J.~B. {Predd}, S.~B. {Kulkarni}, and H.~V. {Poor}, ``Distributed learning in
  wireless sensor networks,'' {\em IEEE Signal Processing Magazine}, vol.~23,
  pp.~56--69, July 2006.

\bibitem{vehicles}
R.~W. {Beard} and V.~{Stepanyan}, ``Information consensus in distributed
  multiple vehicle coordinated control,'' in {\em 42nd IEEE International
  Conference on Decision and Control (IEEE Cat. No.03CH37475)}, vol.~2,
  pp.~2029--2034 Vol.2, Dec 2003.

\bibitem{smart_grids}
G.~B. {Giannakis}, V.~{Kekatos}, N.~{Gatsis}, S.~{Kim}, H.~{Zhu}, and B.~F.
  {Wollenberg}, ``Monitoring and optimization for power grids: A signal
  processing perspective,'' {\em IEEE Signal Processing Magazine}, vol.~30,
  pp.~107--128, Sep. 2013.

\bibitem{bertsekas1989parallel}
D.~P. Bertsekas and J.~N. Tsitsiklis, {\em Parallel and distributed
  computation: numerical methods}, vol.~23.
\newblock Prentice hall Englewood Cliffs, NJ, 1989.

\bibitem{NedicSubgradientConsensus}
A.~Nedi\'c and A.~Ozdaglar, ``Distributed subgradient methods for multi-agent
  optimization,'' {\em IEEE Transactions on Automatic Control}, vol.~54, pp.~48
  --61, Jan 2009.

\bibitem{nedic_constrained_2010}
A.~Nedic, A.~Ozdaglar, and P.~Parrilo, ``Constrained {Consensus} and
  {Optimization} in {Multi}-{Agent} {Networks},'' {\em IEEE Transactions on
  Automatic Control}, vol.~55, pp.~922--938, Apr. 2010.

\bibitem{MouraFastGrad}
D.~Jakovetic, J.~Xavier, and J.~M.~F. Moura, ``{Fast Distributed Gradient
  Methods},'' {\em IEEE Transactions on Automatic Control}, vol.~59, no.~5,
  pp.~1131--1146, 2014.

\bibitem{near_dgd_2017}
A.~S. {Berahas}, R.~{Bollapragada}, N.~{Shirish Keskar}, and E.~{Wei},
  ``{Balancing Communication and Computation in Distributed Optimization},''
  {\em arXiv e-prints}, p.~arXiv:1709.02999, Sep 2017.

\bibitem{extra}
W.~Shi, Q.~Ling, G.~Wu, and W.~Yin, ``Extra: An exact first-order algorithm for
  decentralized consensus optimization,'' {\em SIAM Journal on Optimization},
  vol.~25, no.~2, pp.~944--966, 2015.

\bibitem{zhu_distr_conv_2012}
M.~{Zhu} and S.~{Martinez}, ``On distributed convex optimization under
  inequality and equality constraints,'' {\em IEEE Transactions on Automatic
  Control}, vol.~57, pp.~151--164, Jan 2012.

\bibitem{KIA2015254}
S.~S. Kia, J.~Cortés, and S.~Martínez, ``Distributed convex optimization via
  continuous-time coordination algorithms with discrete-time communication,''
  {\em Automatica}, vol.~55, pp.~254 -- 264, 2015.

\bibitem{di2016next}
P.~Di~Lorenzo and G.~Scutari, ``Next: In-network nonconvex optimization,'' {\em
  IEEE Transactions on Signal and Information Processing over Networks},
  vol.~2, no.~2, pp.~120--136, 2016.

\bibitem{qu2017harnessing}
G.~Qu and N.~Li, ``Harnessing smoothness to accelerate distributed
  optimization,'' {\em IEEE Transactions on Control of Network Systems}, 2017.

\bibitem{diging}
A.~Nedich, A.~Olshevsky, and A.~Shi, ``Achieving geometric convergence for
  distributed optimization over time-varying graphs,'' {\em SIAM Journal on
  Optimization}, vol.~27, pp.~2597--2633, 1 2017.

\bibitem{push_pull}
S.~Pu, W.~Shi, J.~Xu, and A.~Nedić, ``A push-pull gradient method for
  distributed optimization in networks,'' 2018.

\bibitem{ADMMBoyd}
S.~Boyd, N.~Parikh, E.~Chu, B.~Peleato, and J.~Eckstein, ``{Distributed
  Optimization and Statistical Learning via the Alternating Direction Method of
  Multipliers},'' {\em Foundations and Trends in Machine Learning}, vol.~3,
  no.~1, pp.~1--122, 2010.

\bibitem{mokhtari_newton}
A.~{Mokhtari}, Q.~{Ling}, and A.~{Ribeiro}, ``Network newton distributed
  optimization methods,'' {\em IEEE Transactions on Signal Processing},
  vol.~65, pp.~146--161, Jan 2017.

\bibitem{samira_newton}
F.~{Mansoori} and E.~{Wei}, ``Superlinearly convergent asynchronous distributed
  network newton method,'' in {\em 2017 IEEE 56th Annual Conference on Decision
  and Control (CDC)}, pp.~2874--2879, Dec 2017.

\bibitem{tsianos_ml_2012}
K.~I. {Tsianos}, S.~{Lawlor}, and M.~G. {Rabbat}, ``Consensus-based distributed
  optimization: Practical issues and applications in large-scale machine
  learning,'' in {\em 2012 50th Annual Allerton Conference on Communication,
  Control, and Computing (Allerton)}, pp.~1543--1550, Oct 2012.

\bibitem{predd_ml_2009}
J.~B. {Predd}, S.~R. {Kulkarni}, and H.~V. {Poor}, ``A collaborative training
  algorithm for distributed learning,'' {\em IEEE Transactions on Information
  Theory}, vol.~55, pp.~1856--1871, April 2009.

\bibitem{Bekkerman:2011:SUM:2124408}
R.~Bekkerman, M.~Bilenko, and J.~Langford, {\em Scaling Up Machine Learning:
  Parallel and Distributed Approaches}.
\newblock New York, NY, USA: Cambridge University Press, 2011.

\bibitem{Boyd:2011:DOS:2185815.2185816}
S.~Boyd, N.~Parikh, E.~Chu, B.~Peleato, and J.~Eckstein, ``Distributed
  optimization and statistical learning via the alternating direction method of
  multipliers,'' {\em Found. Trends Mach. Learn.}, vol.~3, pp.~1--122, Jan.
  2011.

\bibitem{duchi_dual_2012}
J.~Duchi, A.~Agarwal, and M.~Wainwright, ``Dual {Averaging} for {Distributed}
  {Optimization}: {Convergence} {Analysis} and {Network} {Scaling},'' {\em IEEE
  Transactions on Automatic Control}, vol.~57, pp.~592--606, Mar. 2012.
\newblock arXiv: 1005.2012.

\bibitem{shen_towards_2018}
Z.~Shen, A.~Mokhtari, T.~Zhou, P.~Zhao, and H.~Qian, ``Towards {More}
  {Efficient} {Stochastic} {Decentralized} {Learning}: {Faster} {Convergence}
  and {Sparse} {Communication},'' {\em arXiv:1805.09969 [cs, stat]}, May 2018.
\newblock arXiv: 1805.09969.

\bibitem{sgd_bottou}
L.~Bottou, ``Large-scale machine learning with stochastic gradient descent,''
  in {\em Proceedings of COMPSTAT'2010} (Y.~Lechevallier and G.~Saporta, eds.),
  (Heidelberg), pp.~177--186, Physica-Verlag HD, 2010.

\bibitem{Bach:2011:NAS:2986459.2986510}
F.~Bach and E.~Moulines, ``Non-asymptotic analysis of stochastic approximation
  algorithms for machine learning,'' in {\em Proceedings of the 24th
  International Conference on Neural Information Processing Systems}, NIPS'11,
  (USA), pp.~451--459, Curran Associates Inc., 2011.

\bibitem{sundharram_distributed_2010}
S.~Sundhar Ram, A.~Nedić, and V.~V. Veeravalli, ``Distributed {Stochastic}
  {Subgradient} {Projection} {Algorithms} for {Convex} {Optimization},'' {\em
  Journal of Optimization Theory and Applications}, vol.~147, pp.~516--545,
  Dec. 2010.

\bibitem{bianchi_convergence_2011}
P.~Bianchi and J.~Jakubowicz, ``Convergence of a {Multi}-{Agent} {Projected}
  {Stochastic} {Gradient} {Algorithm} for {Non}-{Convex} {Optimization},'' {\em
  arXiv:1107.2526 [cs, math]}, July 2011.
\newblock arXiv: 1107.2526.

\bibitem{srivastava_distributed_2011}
K.~Srivastava and A.~Nedic, ``Distributed {Asynchronous} {Constrained}
  {Stochastic} {Optimization},'' {\em IEEE Journal of Selected Topics in Signal
  Processing}, vol.~5, pp.~772--790, Aug. 2011.

\bibitem{nedic_stochastic_2014}
A.~Nedic and A.~Olshevsky, ``Stochastic {Gradient}-{Push} for {Strongly}
  {Convex} {Functions} on {Time}-{Varying} {Directed} {Graphs},'' {\em
  arXiv:1406.2075 [cs, math]}, June 2014.
\newblock arXiv: 1406.2075.

\bibitem{towfic_adaptive_2014}
Z.~J. Towfic and A.~H. Sayed, ``Adaptive {Penalty}-{Based} {Distributed}
  {Stochastic} {Convex} {Optimization},'' {\em IEEE Transactions on Signal
  Processing}, vol.~62, pp.~3924--3938, Aug. 2014.

\bibitem{shamir_distributed_2014}
O.~Shamir and N.~Srebro, ``Distributed stochastic optimization and learning,''
  in {\em 2014 52nd {Annual} {Allerton} {Conference} on {Communication},
  {Control}, and {Computing} ({Allerton})}, (Monticello, IL, USA),
  pp.~850--857, IEEE, Sept. 2014.

\bibitem{mokhtari_dsa:_2015}
A.~Mokhtari and A.~Ribeiro, ``{DSA}: {Decentralized} {Double} {Stochastic}
  {Averaging} {Gradient} {Algorithm},'' {\em arXiv:1506.04216 [math]}, June
  2015.
\newblock arXiv: 1506.04216.

\bibitem{vanli_stochastic_2014}
N.~D. Vanli, M.~O. Sayin, and S.~S. Kozat, ``Stochastic {Subgradient}
  {Algorithms} for {Strongly} {Convex} {Optimization} over {Distributed}
  {Networks},'' {\em arXiv:1409.8277 [cs, math]}, Sept. 2014.
\newblock arXiv: 1409.8277.

\bibitem{sirb_decentralized_2016}
B.~Sirb and X.~Ye, ``Decentralized {Consensus} {Algorithm} with {Delayed} and
  {Stochastic} {Gradients},'' {\em arXiv:1604.05649 [math]}, Apr. 2016.
\newblock arXiv: 1604.05649.

\bibitem{chatzipanagiotis_distributed_2016}
N.~Chatzipanagiotis and M.~M. Zavlanos, ``A {Distributed} {Algorithm} for
  {Convex} {Constrained} {Optimization} {Under} {Noise},'' {\em IEEE
  Transactions on Automatic Control}, vol.~61, pp.~2496--2511, Sept. 2016.

\bibitem{lan_communication-efficient_2017}
G.~Lan, S.~Lee, and Y.~Zhou, ``Communication-{Efficient} {Algorithms} for
  {Decentralized} and {Stochastic} {Optimization},'' {\em arXiv:1701.03961 [cs,
  math]}, Jan. 2017.
\newblock arXiv: 1701.03961.

\bibitem{hong_stochastic_2017}
M.~Hong and T.-H. Chang, ``Stochastic {Proximal} {Gradient} {Consensus} {Over}
  {Random} {Networks},'' {\em IEEE Transactions on Signal Processing}, vol.~65,
  pp.~2933--2948, June 2017.

\bibitem{morral_success_2017}
G.~Morral, P.~Bianchi, and G.~Fort, ``Success and {Failure} of
  {Adaptation}-{Diffusion} {Algorithms} {With} {Decaying} {Step} {Size} in
  {Multiagent} {Networks},'' {\em IEEE Transactions on Signal Processing},
  vol.~65, pp.~2798--2813, June 2017.

\bibitem{pu_flocking-based_2017}
S.~Pu and A.~Garcia, ``A {Flocking}-based {Approach} for {Distributed}
  {Stochastic} {Optimization},'' {\em arXiv:1709.07085 [math]}, Sept. 2017.
\newblock arXiv: 1709.07085.

\bibitem{bijral_data-dependent_2016}
A.~Bijral, A.~D. Sarwate, and N.~Srebro, ``Data-dependent bounds on network
  gradient descent,'' in {\em 2016 54th {Annual} {Allerton} {Conference} on
  {Communication}, {Control}, and {Computing} ({Allerton})}, (Monticello, IL,
  USA), pp.~869--874, IEEE, Sept. 2016.

\bibitem{lian_can_nodate}
X.~Lian, C.~Zhang, H.~Zhang, C.-J. Hsieh, W.~Zhang, and J.~Liu, ``Can
  {Decentralized} {Algorithms} {Outperform} {Centralized} {Algorithms}? {A}
  {Case} {Study} for {Decentralized} {Parallel} {Stochastic} {Gradient}
  {Descent},'' p.~11.

\bibitem{lian_asynchronous_2017}
X.~Lian, W.~Zhang, C.~Zhang, and J.~Liu, ``Asynchronous {Decentralized}
  {Parallel} {Stochastic} {Gradient} {Descent},'' {\em arXiv:1710.06952 [cs,
  math, stat]}, Oct. 2017.
\newblock arXiv: 1710.06952.

\bibitem{nokleby_distributed_2017}
M.~Nokleby and W.~U. Bajwa, ``Distributed mirror descent for stochastic
  learning over rate-limited networks,'' in {\em 2017 {IEEE} 7th
  {International} {Workshop} on {Computational} {Advances} in {Multi}-{Sensor}
  {Adaptive} {Processing} ({CAMSAP})}, (Curacao), pp.~1--5, IEEE, Dec. 2017.

\bibitem{pu_distributed_2018-1}
S.~Pu and A.~Nedić, ``Distributed {Stochastic} {Gradient} {Tracking}
  {Methods},'' {\em arXiv:1805.11454 [cs, math, stat]}, May 2018.
\newblock arXiv: 1805.11454.

\bibitem{jakovetic_convergence_2018}
D.~Jakovetic, D.~Bajovic, A.~K. Sahu, and S.~Kar, ``Convergence rates for
  distributed stochastic optimization over random networks,'' {\em
  arXiv:1803.07836 [math]}, Mar. 2018.
\newblock arXiv: 1803.07836.

\bibitem{tang_d$^2$:_2018}
H.~Tang, X.~Lian, M.~Yan, C.~Zhang, and J.~Liu, ``D\${\textasciicircum}2\$:
  {Decentralized} {Training} over {Decentralized} {Data},'' {\em
  arXiv:1803.07068 [cs, stat]}, Mar. 2018.
\newblock arXiv: 1803.07068.

\bibitem{sahu_communication-efficient_2018}
A.~K. Sahu, D.~Jakovetic, D.~Bajovic, and S.~Kar, ``Communication-{Efficient}
  {Distributed} {Strongly} {Convex} {Stochastic} {Optimization}:
  {Non}-{Asymptotic} {Rates},'' {\em arXiv:1809.02920 [math]}, Sept. 2018.
\newblock arXiv: 1809.02920.

\bibitem{olshevsky_robust_2018}
A.~Olshevsky, I.~C. Paschalidis, and A.~Spiridonoff, ``Robust {Asynchronous}
  {Stochastic} {Gradient}-{Push}: {Asymptotically} {Optimal} and
  {Network}-{Independent} {Performance} for {Strongly} {Convex} {Functions},''
  {\em arXiv:1811.03982 [math]}, Nov. 2018.
\newblock arXiv: 1811.03982.

\bibitem{yuan_variance-reduced_2019}
K.~Yuan, B.~Ying, J.~Liu, and A.~H. Sayed, ``Variance-{Reduced} {Stochastic}
  {Learning} by {Networked} {Agents} {Under} {Random} {Reshuffling},'' {\em
  IEEE Transactions on Signal Processing}, vol.~67, pp.~351--366, Jan. 2019.

\bibitem{xin_distributed_2019}
R.~Xin, A.~K. Sahu, U.~A. Khan, and S.~Kar, ``Distributed stochastic
  optimization with gradient tracking over strongly-connected networks,'' {\em
  arXiv:1903.07266 [cs, stat]}, Mar. 2019.
\newblock arXiv: 1903.07266.

\bibitem{olshevsky_non-asymptotic_2019}
A.~Olshevsky, I.~C. Paschalidis, and S.~Pu, ``A {Non}-{Asymptotic} {Analysis}
  of {Network} {Independence} for {Distributed} {Stochastic} {Gradient}
  {Descent},'' {\em arXiv:1906.02702 [cs, math]}, June 2019.
\newblock arXiv: 1906.02702.

\bibitem{nesterov_introductory_1998}
Y.~Nesterov, ``Introductory {Lectures} on {Convex} {Programming} {Volume} {I}:
  {Basic} course,'' p.~212, July 1998.

\bibitem{mushrooms}
D.~Dua and C.~Graff, ``{UCI} machine learning repository,'' 2017.

\end{thebibliography}

\end{document}